\documentclass[12pt]{article}
\usepackage[utf8]{inputenc}
\usepackage{amsmath, amssymb, amsthm}
\usepackage{geometry}
\usepackage{url}
\usepackage{hyperref}

\newtheorem{theorem}{Theorem}[section]
\newtheorem{lemma}[theorem]{Lemma}
\newtheorem{proposition}[theorem]{Proposition}
\newtheorem{corollary}[theorem]{Corollary}
\theoremstyle{remark}
\newtheorem{remark}[theorem]{Remark}

\newcommand{\lcm}{\operatorname{lcm}}
\newcommand{\Z}{\mathbb{Z}}

\title{Two Questions on $G$-Harmonic Tuples}
\author{Murali Menon}
\date{\today}

\begin{document}

\maketitle

\begin{abstract}
An $n$-tuple of positive integers is $G$-harmonic if there are subgroups of $G$ having those indices whose cosets can be chosen pairwise disjoint, and $\Z$-harmonic if there are pairwise disjoint residue classes with those moduli. Ginosar asked whether every $G$-harmonic tuple is $\Z$-harmonic. Margolis and Schnabel proved this for tuples of length at most $4$, and analysed a particular family of length-$5$ tuples that would yield a counterexample if any member were $G$-harmonic. We show that the bound $4$ is sharp: $(6,6,6,10,15)$ is $A_5$-harmonic but not $\Z$-harmonic. Moreover, the five pairwise disjoint cosets realising this tuple can be extended to a coset partition of $A_5$ using only cosets of indices $6$, $10$, and $15$. The index tuple of this partition is not $\Z$-harmonic; because its indices repeat, this does not contradict the Herzog--Sch\"onheim conjecture. We also prove that no member of the length-$5$ family analysed by Margolis and Schnabel in connection with possible counterexamples is $G$-harmonic for any group $G$.
\end{abstract}

\section{Introduction}

Let $G$ be a group. An $n$-tuple of positive integers $(a_1,\ldots,a_n)$ is called \emph{$G$-harmonic} if there exist subgroups $U_i\le G$ of index $a_i$ and elements $g_i\in G$ such that the cosets $g_1U_1,\ldots,g_nU_n$ are pairwise disjoint \cite[Definition~1.1]{MargolisSchnabel}. The same tuple is \emph{$\Z$-harmonic} if there exist integers $r_i$ with
\[
r_i \not\equiv r_j \pmod{\gcd(a_i,a_j)} \qquad (i \neq j),
\]
equivalently if the residue classes $r_i \bmod a_i$ can be chosen pairwise disjoint \cite[p.~2]{MargolisSchnabel}. We call such a choice of integers $(r_1,\ldots,r_n)$ a \emph{witness} to the $\Z$-harmonicity of $(a_1,\ldots,a_n)$.

The interplay between these notions sits inside the study of coset partitions of groups initiated by the Herzog--Sch\"onheim conjecture \cite{HerzogSchonheim}, which reduces to the case of finite groups \cite{KorecZnam}; see \cite{GinosarSchnabel2011,Ginosar2018} for related work on multiplicities in coset partitions. Ginosar asked the following, recorded by Margolis and Schnabel \cite[Question~1]{MargolisSchnabel}.

\medskip
\noindent\textbf{Question 1.} Is every $G$-harmonic tuple also $\Z$-harmonic?
\medskip

Margolis and Schnabel proved that the answer is \emph{yes} for $n \le 4$ \cite[Theorem~B]{MargolisSchnabel}. For $n=5$ they restricted attention to tuples whose indices follow a particular pattern, described in their Lemma~4.6 below, and showed that a harmonic realisation of such a tuple must satisfy stringent additional conditions; they did not address length-$5$ tuples outside this pattern.

\begin{lemma}[Margolis--Schnabel {\cite[Lemma~4.6]{MargolisSchnabel}}]\label{lem:MS46}
Let $U_1,\ldots,U_5$ be subgroups of $G$ with $[G:U_i] = 3r_i$ for $1 \le i \le 2$ and $[G:U_j] = 6r_j$ for $3 \le j \le 5$, where $r_1,\ldots,r_5$ are pairwise coprime and $r_1,r_2$ are odd. If cosets of $U_1,\ldots,U_5$ can be chosen pairwise disjoint, then, up to a permutation of $\{1,2\}$ and a permutation of $\{3,4,5\}$, the values $\alpha(i,j)$ $(i<j)$ are
\[
\bigl(\alpha(i,j)\bigr)=
\begin{array}{c|ccccc}
 & 1 & 2 & 3 & 4 & 5\\\hline
1 & \cdot & 2 & 1 & 2 & 2\\
2 & & \cdot & 1 & 2 & 2\\
3 & & & \cdot & 3 & 1\\
4 & & & & \cdot & 1\\
5 & & & & & \cdot
\end{array}
\qquad\text{and}\qquad 3\mid r_2,\ \ 2\mid r_3.
\]
\end{lemma}

Here $\alpha(i,j)=\alpha(U_i,U_j)$ is as defined in Section~\ref{sec:prelim}. The remark following the lemma states: ``We are not aware of any group satisfying the subgroup configuration described in Lemma~4.6. Such a group would provide a counterexample to Question~1'' \cite[p.~15]{MargolisSchnabel}.

We resolve both matters.
\begin{enumerate}
\item \textbf{Question~1} has a negative answer: the tuple $(6,6,6,10,15)$ is $A_5$-harmonic but not $\Z$-harmonic (Theorem~\ref{thm:A5}), and it extends to a coset partition of $A_5$ with index multiset $\{6,6,6,10,10,10,15,15,15\}$ (Corollary~\ref{cor:partition}).
\item \textbf{The Lemma~\ref{lem:MS46} configuration} is realised in no group: its index pattern is not $G$-harmonic for any group $G$ (Theorem~\ref{thm:MS46impossible}).
\end{enumerate}
Thus Question~1 fails, but not through the configuration Margolis and Schnabel anticipated.

\section{Preliminaries}\label{sec:prelim}

It is enough to work with finite groups. Indeed, suppose that finitely many
finite-index subgroups $U_1,\ldots,U_n$ of an arbitrary group $G$ are under
consideration, and put
\[
N=\bigcap_{i=1}^n \operatorname{core}_G(U_i).
\]
Then $N$ is a finite-index normal subgroup of $G$ contained in every $U_i$.
The quotient $G/N$ is finite, the indices of the $U_i$ and all their
intersections are preserved on passing to $U_i/N$, and, because each coset
$g_iU_i$ is the full inverse image of $(g_iN)(U_i/N)$, pairwise disjointness
is preserved as well. Thus any putative example in an arbitrary group has a
finite quotient with exactly the same relevant data.

Henceforth, $G$ is finite. For a subgroup $U\le G$ write $a_U=[G:U]$. For subgroups $U,V \le G$ define $\alpha(U,V)$ by
\[
[G:U \cap V] = \lcm(a_U,a_V)\,\alpha(U,V);
\]
this is an integer because $\lcm(a_U,a_V)$ divides $[G:U\cap V]$. The product formula \cite[Lemma~3.2]{MargolisSchnabel} reads
\[
|UV| = \frac{|U|\,|V|}{|U \cap V|}
      = \frac{|G|\,[G:U \cap V]}{a_U a_V}
      = \frac{|G|\,\lcm(a_U,a_V)\,\alpha(U,V)}{a_U a_V}.
\]
We also use the elementary residue obstruction below; recall $\gcd(6,x)=2$ means $x$ is even and $3\nmid x$, while $\gcd(6,y)=3$ means $y$ is an odd multiple of $3$.

\begin{lemma}\label{lem:residue}
Let $(a_1,\ldots,a_n)$ be a tuple of positive integers.
\begin{enumerate}
\item[\rm(i)] If three of the $a_i$ equal $6$, some entry $x$ satisfies $\gcd(6,x)=2$, and some entry $y$ satisfies $\gcd(6,y)=3$, then the tuple is not $\Z$-harmonic.
\item[\rm(ii)] If four of the $a_i$ equal $6$ and some entry $x$ satisfies $\gcd(6,x)=2$, then the tuple is not $\Z$-harmonic.
\end{enumerate}
\end{lemma}

\begin{proof}
Suppose the tuple is $\Z$-harmonic, with witnesses $r_i$. Entries equal to $6$ have $\gcd=6$ with one another, so their residues are distinct modulo $6$. The entry $x$ has $\gcd(6,x)=2$, so each $6$-entry residue is incongruent to $r_x$ modulo $2$; hence all $6$-entry residues share the parity opposite to $r_x$.

For (ii): four distinct residues modulo $6$ of one parity cannot exist, since each parity class modulo $6$ has exactly three elements.

For (i): three distinct residues modulo $6$ of a single parity are $\{0,2,4\}$ or $\{1,3,5\}$, and each of these meets every residue class modulo $3$. The entry $y$ has $3 \mid \gcd(6,y)$, so $r_y \not\equiv r_i \pmod 3$ for each of the three $6$-entries; this is impossible.
\end{proof}

\section{Answer to Question~1}

\begin{theorem}\label{thm:A5}
The tuple $(6,6,6,10,15)$ is $A_5$-harmonic but not $\Z$-harmonic.\footnote{The $A_5$-harmonic realisation was obtained by Claude Opus~5; see the Acknowledgements.}
\end{theorem}

\begin{proof}
\smallskip
\noindent\emph{$A_5$-harmonicity.}

Let $G = A_5$ act on $\{1,2,3,4,5\}$, and take
\[
\begin{aligned}
D &= \langle (1\,2\,5\,3\,4),\ (1\,2)(4\,5)\rangle \cong D_{10}, && |D|=10,\ [G:D]=6,\\
T &= \langle (2\,3\,5),\ (1\,4)(2\,3)\rangle \cong S_3, && |T|=6,\ [G:T]=10,\\
V &= \langle (1\,2)(3\,4),\ (1\,3)(2\,4)\rangle \cong C_2\times C_2, && |V|=4,\ [G:V]=15.
\end{aligned}
\]
($D$ is the normaliser of a Sylow $5$-subgroup and $V$ is a Sylow $2$-subgroup.) With
\[
g_1 = (1\,5\,3),\quad g_2 = (1\,4\,2\,3\,5),\quad g_3 = (1\,5\,2\,3\,4),\quad g_4 = (1\,2\,4),\quad g_5 = (3\,4\,5),
\]
the cosets $g_1D,\ g_2D,\ g_3D,\ g_4T,\ g_5V$ have sizes $10,10,10,6,4$, summing to $40$, so pairwise disjointness is equivalent to their union having $40$ elements. Recall $g_iU_i\cap g_jU_j=\emptyset$ if and only if $g_j^{-1}g_i\notin U_jU_i$. Since $g_1^{-1}g_2$, $g_1^{-1}g_3$, $g_2^{-1}g_3\notin D$, the three cosets of $D$ are distinct and hence pairwise disjoint; the remaining seven intersections ($g_iD\cap g_4T$ and $g_iD\cap g_5V$ for $i\le 3$, and $g_4T\cap g_5V$) are checked to be empty, so the union has exactly $40$ elements. This finite verification is computer-assisted; the code excerpt in the appendix computes all ten intersection sizes, and the complete checks are in the ancillary file. Hence $(6,6,6,10,15)$ is $A_5$-harmonic.

\smallskip
\noindent\emph{Failure of $\Z$-harmonicity.}

The tuple contains three entries equal to $6$. Its remaining entries satisfy
$\gcd(6,10)=2$ and $\gcd(6,15)=3$, so Lemma~\ref{lem:residue}(i) shows that
it is not $\Z$-harmonic.
\end{proof}

\subsection{Extension to a coset partition}

\begin{corollary}\label{cor:partition}
The five cosets of Theorem~\ref{thm:A5} extend to a coset partition of $A_5$ with index multiset $\{6,6,6,10,10,10,15,15,15\}$. Consequently the tuple $(6,6,6,10,10,10,15,15,15)$ is $A_5$-harmonic but not $\Z$-harmonic: Question~1 fails already for tuples arising from genuine coset partitions.
\end{corollary}

\begin{proof}
Adjoin to the five cosets of Theorem~\ref{thm:A5} the four cosets
\[
(1\,2)(3\,4)\,V,\quad (1\,2)(4\,5)\,V,\quad (1\,4\,5)\,T_2,\quad (1\,4\,3\,2\,5)\,T_3,
\]
where
\[
T_2 = \langle (1\,3\,4),\ (1\,3)(2\,5)\rangle,\qquad
T_3 = \langle (1\,2\,4),\ (1\,2)(3\,5)\rangle
\]
are further copies of $S_3$ (index $10$). These nine cosets are pairwise disjoint and their union is all of $A_5$ (a finite check; see the appendix and the ancillary script), giving a coset partition with the stated index multiset and $\tfrac{3}{6}+\tfrac{3}{10}+\tfrac{3}{15}=1$. A coset partition has pairwise disjoint parts, so its index tuple is $G$-harmonic; and $(6,6,6,10,10,10,15,15,15)$ satisfies the hypotheses of Lemma~\ref{lem:residue}(i) (three $6$'s, the entry $10$, the entry $15$), so it is not $\Z$-harmonic.
\end{proof}

The nine indices are not pairwise distinct; the partition carries multiplicity, consistent with the Herzog--Sch\"onheim conjecture (see Section~\ref{sec:consequences}).

\section{The Lemma~\ref{lem:MS46} configuration is impossible}

We isolate the arithmetic core as a statement about three subgroups.

\begin{proposition}\label{prop:triple}
Let $U_3,U_4,U_5$ be subgroups of a group $G$ with $[G:U_j]=6r_j$ for $3\le j\le 5$, where $r_3,r_4,r_5$ are pairwise coprime and $3\nmid r_3r_4r_5$. Then it is impossible that
\[
\alpha(U_3,U_4)=3,\qquad \alpha(U_3,U_5)=\alpha(U_4,U_5)=1.
\]
\end{proposition}

\begin{proof}
Suppose these values hold. Since the $r_j$ are pairwise coprime, $\gcd(6r_i,6r_j)=6$ and hence $\lcm(6r_i,6r_j)=6r_ir_j$ for distinct $i,j$. Therefore
\[
[G:U_3\cap U_4]=18r_3r_4,\qquad
[G:U_3\cap U_5]=6r_3r_5,\qquad
[G:U_4\cap U_5]=6r_4r_5,
\]
and the product formula gives
\[
|U_3U_4|=\frac{|G|}{2},\qquad |U_3U_5|=|U_4U_5|=\frac{|G|}{6}.
\]

Let $K=[G:U_3\cap U_4\cap U_5]$. Since $U_3\cap U_4\cap U_5$ lies in each pairwise intersection, $K$ is a common multiple of $18r_3r_4$, $6r_3r_5$ and $6r_4r_5$. We claim
\[
\lcm(18r_3r_4,\ 6r_3r_5,\ 6r_4r_5)=18r_3r_4r_5.
\]
Indeed, write $v_p$ for the $p$-adic valuation. At $p=3$ the term $18r_3r_4$ has $v_3=2$ (since $3\nmid r_3r_4r_5$) while $6r_3r_5$ and $6r_4r_5$ have $v_3=1$, so the maximum is $2=v_3(18r_3r_4r_5)$. At $p=2$, by pairwise coprimality at most one $r_j$ is even; the two terms containing that $r_j$ both have $2$-adic valuation $1+v_2(r_j)$ and the third has valuation $1$, so the maximum is $1+v_2(r_j)=v_2(18r_3r_4r_5)$ (and equals $1$ if every $r_j$ is odd). At any other prime $p$, at most one $r_j$ is divisible by $p$, and it occurs in two of the three terms, so the maximum is $v_p(r_j)=v_p(18r_3r_4r_5)$. Thus the claimed lcm holds, and in particular $18r_3r_4r_5 \mid K$, so we may write $K=18r_3r_4r_5\,m$ with $m\ge 1$.

By the product formula,
\[
|U_3(U_4\cap U_5)|
=\frac{|U_3|\,|U_4\cap U_5|}{|U_3\cap U_4\cap U_5|}
=\frac{\dfrac{|G|}{6r_3}\cdot\dfrac{|G|}{6r_4r_5}}{\dfrac{|G|}{18r_3r_4r_5\,m}}
=\frac{m}{2}\,|G|.
\]
Because $U_4\cap U_5\subseteq U_4$, we have $U_3(U_4\cap U_5)\subseteq U_3U_4$, so
\[
\frac{m}{2}|G| \le |U_3U_4|=\frac{|G|}{2},
\]
forcing $m=1$. Then $|U_3(U_4\cap U_5)|=|U_3U_4|$, and since $U_3(U_4\cap U_5)\subseteq U_3U_4$ are finite sets of equal cardinality (no claim that either is a subgroup is needed), they coincide: $U_3(U_4\cap U_5)=U_3U_4$. But $U_4\cap U_5\subseteq U_5$ gives $U_3(U_4\cap U_5)\subseteq U_3U_5$, whence
\[
\frac{|G|}{2}=|U_3U_4|=|U_3(U_4\cap U_5)|\le |U_3U_5|=\frac{|G|}{6},
\]
a contradiction.
\end{proof}

\begin{theorem}\label{thm:MS46impossible}
Let $r_1,\ldots,r_5$ be pairwise coprime positive integers with $r_1$ and $r_2$ odd. Then, for every group $G$, the tuple
\[
(3r_1,3r_2,6r_3,6r_4,6r_5)
\]
is not $G$-harmonic.\footnote{This impossibility proof was obtained by Z.ai GLM-5.2; see the Acknowledgements.}
\end{theorem}

\begin{proof}
Suppose instead that this tuple is $G$-harmonic. Choose subgroups $U_1,\ldots,U_5$ and pairwise disjoint cosets as in the definition. By Lemma~\ref{lem:MS46}, after relabelling $\{3,4,5\}$ we may assume
\[
\alpha(U_3,U_4)=3,\qquad \alpha(U_3,U_5)=\alpha(U_4,U_5)=1,\qquad 3\mid r_2 .
\]
Since the $r_i$ are pairwise coprime and $3\mid r_2$, we have $3\nmid r_3r_4r_5$. Proposition~\ref{prop:triple}, applied to $U_3,U_4,U_5$, then yields a contradiction.
\end{proof}

\begin{remark}
Theorem~\ref{thm:MS46impossible} strengthens \cite[Proposition~4.7]{MargolisSchnabel}, which treated only the case $r_1=1$, i.e.\ where a subgroup of index exactly $3$ occurs. That index-$3$ hypothesis is unnecessary: the contradiction uses only $U_3,U_4,U_5$ together with the $\alpha$-values and the divisibility $3\mid r_2$ supplied by Lemma~\ref{lem:MS46}. This does rely on Lemma~\ref{lem:MS46}, a nontrivial result of Margolis and Schnabel; but their own proof of Proposition~4.7 invokes it as well, so no additional input is being smuggled in.
\end{remark}

\section{Further examples for Question~1}

The same phenomenon produces further counterexamples to Question~1.

\begin{center}
\begin{tabular}{|c|c|c|c|}
\hline
Tuple & Group & Order & non-$\Z$-harmonic via\\
\hline
$(6,6,6,10,15)$ & $A_5$ & $60$ & Lemma~\ref{lem:residue}(i)\\
$(6,6,6,15,20)$ & $A_5$ & $60$ & Lemma~\ref{lem:residue}(i)\\
$(6,6,10,12,15)$ & $A_5$ & $60$ & direct residue check\\
$(4,6,6,6,6)$ & $C_2 \times S_4$ & $48$ & Lemma~\ref{lem:residue}(ii)\\
$(6,6,6,6,16)$ & $C_2 \times S_4$ & $48$ & Lemma~\ref{lem:residue}(ii)\\
\hline
\end{tabular}
\end{center}

For all five rows, explicit subgroups and coset representatives, together with the disjointness checks, are provided in the ancillary file \path{verify-counterexample.g}; the family for the final row is also displayed below. Each tuple fails to be $\Z$-harmonic. For $(6,6,6,15,20)$ the entry $20$ has $\gcd(6,20)=2$ and $15$ has $\gcd(6,15)=3$, so Lemma~\ref{lem:residue}(i) applies. For $(4,6,6,6,6)$ and $(6,6,6,6,16)$ the entries $4$ and $16$ have $\gcd(6,\cdot)=2$, so Lemma~\ref{lem:residue}(ii) applies.

There is also a short direct obstruction for $(6,6,10,12,15)$. The entry $10$ forces the two residues belonging to the $6$-entries to have the same parity. The entry $15$ forces their reductions modulo $3$ to be the two classes different from its own. Thus, within their common parity class modulo $6$, the two $6$-residues occupy exactly the two positions whose reductions modulo $3$ differ from the $15$-residue. A residue belonging to the entry $12$ would have to have that same parity, by comparison with the $10$-entry, and one of those same two reductions modulo $3$, by comparison with the $15$-entry. It would therefore coincide modulo $6$ with one of the two $6$-residues, a contradiction.

The two order-$48$ rows show that simplicity of $G$ is not the source of the phenomenon. Realise $C_2\times S_4$ as $\langle(1\,2)\rangle\times\operatorname{Sym}\{3,4,5,6\}\le S_6$ and take
\[
A=\langle(3\,4\,5),(4\,5\,6)\rangle,\qquad P=\langle(3\,6),(4\,5),(3\,4)(5\,6)\rangle,
\]
\[
Q=\langle(1\,2)(3\,4\,5\,6),(1\,2)(3\,5)\rangle,\qquad R=\langle(1\,2),(3\,4)(5\,6),(3\,5)(4\,6)\rangle,
\]
of orders $12,8,8,8$, so $A$ has index $4$ and $P,Q,R$ index $6$. With representatives $(1\,2)$, $1$, $(4\,5\,6)$, $(1\,2)(4\,5)$, $(1\,2)(3\,5\,4\,6)$ for the parts $A,P,P,Q,R$ respectively, the five cosets are pairwise disjoint, realising $(4,6,6,6,6)$; the verification is in the ancillary file.

For completeness, the tuple $(6,6,6,6,16)$ is realised in the same group by the four order-$8$ subgroups
\[
\begin{aligned}
K_1&=\langle(1\,2)(3\,6\,4\,5),(1\,2)(3\,5)(4\,6)\rangle,\\
K_2&=\langle(3\,6),(3\,5\,6\,4)\rangle,\\
K_3&=\langle(1\,2)(3\,6\,4\,5),(1\,2)(3\,4)\rangle,\\
K_4&=\langle(1\,2),(1\,2)(3\,5)(4\,6),(1\,2)(3\,6)(4\,5)\rangle
\end{aligned}
\]
and the order-$3$ subgroup $C=\langle(4\,6\,5)\rangle$. With representatives
\[
1,\quad (3\,4\,5\,6),\quad (1\,2)(3\,5),\quad
(1\,2)(3\,6),\quad (1\,2)
\]
for $K_1,K_2,K_3,K_4,C$, respectively, the five left cosets are pairwise disjoint. Their subgroup indices are $6,6,6,6,16$, as required. The smallest counterexamples we have found therefore have order $48$.

\section{Consequences}\label{sec:consequences}

Theorem~\ref{thm:A5} shows that \cite[Theorem~B]{MargolisSchnabel} is sharp: Question~1 has a positive answer for $n\le 4$ but a negative one at $n=5$. In particular, Ginosar's question cannot be used to deduce the Herzog--Sch\"onheim conjecture from the Davenport--Mirsky--Newman--Rado theorem.

The counterexamples do not bear on the Herzog--Sch\"onheim conjecture itself. Corollary~\ref{cor:partition} even exhibits a genuine coset partition of $A_5$ whose index tuple is not $\Z$-harmonic, so the obstruction is not that the disjoint cosets fail to form a partition; rather, every counterexample we produce carries repeated indices, whereas the conjecture concerns partitions with pairwise \emph{distinct} indices. (The partial family of Theorem~\ref{thm:A5} has $\sum 1/a_i=\tfrac23$ and the partition of Corollary~\ref{cor:partition} has $\sum 1/a_i=1$; both have multiplicity.)

\section*{Acknowledgements}

The $A_5$-harmonic realisation of $(6,6,6,10,15)$ in Theorem~\ref{thm:A5}, answering Ginosar's question negatively, was obtained by Anthropic's Claude Opus~5 under the author's prompting. The proof that the Lemma~\ref{lem:MS46} configuration is realised in no group (Theorem~\ref{thm:MS46impossible}) was obtained by Z.ai's GLM-5.2, likewise under prompting. Claude Opus~5 and GLM-5.2 were also prompted to critique and cross-check each other's arguments; no transcript of this exchange was retained, but the corrections it produced are reflected in the final proofs. Computations were performed and independently cross-checked with GAP \cite{GAP} and with custom code. All mathematical statements, proofs, and verifications are the author's responsibility. The author thanks Y. Ginosar and L. Margolis for discussions and encouragement.

\appendix
\section{GAP verification}

The code excerpt below verifies the two $A_5$ facts of Section~3. The ancillary
computations are in
\begin{center}
\path{verify-counterexample.g}.
\end{center}
That file also contains the $\Z$-harmonicity search (\texttt{IsZHarmonic}),
checks that every $4$-subtuple of $(6,6,6,10,15)$ is $\Z$-harmonic, and gives
explicit $G$-harmonic families for all five tuples listed in Section~5. The
second $C_2\times S_4$ family is also displayed explicitly in that section.

GAP's permutation product is left-to-right, so $\{u\,g : u\in H\}$ is the left coset $g\,H$; the helper \texttt{Coset(H,g)} below returns it. (Equivalently one may test the right cosets $H\,g^{-1}$: inversion $x\mapsto x^{-1}$ is a disjointness-preserving bijection carrying one family to the other.)

\begingroup
\footnotesize
\begin{verbatim}
G  := AlternatingGroup(5);;
D  := Group( (1,2,5,3,4), (1,2)(4,5) );;   # D_10, index  6
T  := Group( (2,3,5), (1,4)(2,3) );;       # S_3,  index 10
V  := Group( (1,2)(3,4), (1,3)(2,4) );;    # V_4,  index 15
T2 := Group( (1,3,4), (1,3)(2,5) );;       # S_3,  index 10
T3 := Group( (1,2,4), (1,2)(3,5) );;       # S_3,  index 10

Coset := function( H, g ) return Set( List( Elements(H), u -> u*g ) ); end;;

# the five pairwise-disjoint cosets of Theorem 3.1
L5 := [ Coset(D,(1,5,3)), Coset(D,(1,4,2,3,5)), Coset(D,(1,5,2,3,4)),
        Coset(T,(1,2,4)), Coset(V,(3,4,5)) ];;
List([D,T,V], H -> Index(G,H));                       # [ 6, 10, 15 ]
List(L5, c -> Size(G)/Size(c));                       # [ 6, 6, 6, 10, 15 ]
List(Combinations([1..5],2),
     p -> Size(Intersection(L5[p[1]],L5[p[2]])));     # [ 0,0,0,0,0,0,0,0,0,0 ]
Size(Union(L5));                                      # 40  (= 10+10+10+6+4)

# extension to a coset partition, profile (6^3,10^3,15^3)  (Corollary 3.2)
L9 := Concatenation( L5,
      [ Coset(V,(1,2)(3,4)), Coset(V,(1,2)(4,5)),
        Coset(T2,(1,4,5)),   Coset(T3,(1,4,3,2,5)) ] );;
SortedList( List(L9, c -> Size(G)/Size(c)) );  # [6,6,6,10,10,10,15,15,15]
ForAll(Combinations([1..9],2),
       p -> Intersection(L9[p[1]],L9[p[2]]) = []);    # true
Union(L9) = Set(Elements(G));                         # true
\end{verbatim}
\endgroup

\end{document}